\documentclass[12pt]{article}

\usepackage{graphicx}
\usepackage{amsmath,amsthm}
\usepackage{amsfonts}

\usepackage[all]{xy}
\usepackage{amssymb}
\usepackage[a4paper]{geometry}
\usepackage[lofdepth,lotdepth]{subfig}
\usepackage{graphicx}	
\usepackage[utf8]{inputenc}
\usepackage{mathtools}
\usepackage{pgf,tikz,pgfplots}
\usepackage{bbm}
\usepackage{enumitem}
\pgfplotsset{compat=1.17}

\usepackage[T1]{fontenc}
\usepackage[utf8]{inputenc}
\usepackage[english]{babel}

\usepackage{libertine}
\usepackage{libertinust1math}
\usepackage[export]{adjustbox}

\newtheorem{thm}{Theorem}

\newtheorem{xthm}{Theorem}
\newtheorem{prop}{Proposition}[section]
\newtheorem{lema}{Lemma}[section]

\newtheorem{defn}{Definition}[section]
\newtheorem{rmk}{Remark}[section]

\usepackage{hyperref}

\usepackage{mathrsfs} 

\newcommand{\floork}{\left[\frac{k-1}{2}\right]}
\newcommand{\floort}{\left[\frac{\tau_2-1}{2}\right]}

\newcommand{\rr}{\mathbb{R}}
\newcommand{\zz}{\mathbb{Z}}
\newcommand{\nn}{\mathbb{N}}
\newcommand{\TT}{\mathbb{T}}

\begin{document}
\begin{center}
    \Large{Existence of robust non-uniformly hyperbolic endomorphism in homotopy classes}\\
    \vspace{0.5cm}
    \large{Victor Janeiro}\\
\end{center}
\begin{center}
    \href{mailto:victorgjaneiro@gmail.com}{victorgjaneiro@gmail.com}, ICEx-UFMG, Belo Horizonte-MG, Brazil. 
\end{center}

\begin{abstract}
    We extend the results of \protect\cite{NUH} by showing that any homothety in $\TT^2$ is homotopic to a non-uniformly hyperbolic ergodic area preserving map, provided that its degree is at least $5^2$. We also address other small topological degree cases not considered in the previous article. This proves the existence of a $\mathcal C^1$ open set of non-uniformly hyperbolic systems, that intersects essentially every homotopy classes in $\mathbb T^2$, where the Lyapunov exponents vary continuously. 
\end{abstract}

\section{Introduction}

We study conservative maps of the two-torus $\TT^2$ from the point of view of smooth ergodic theory. We are interested in the Lyapunov exponents of these systems, in particular, in extending the results obtained in \protect\cite{NUH} to the homothety case and some cases with lower topological degree, which were not included in the previous results. With this in mind, some familiarity with the results of \protect\cite{NUH} is desirable. 

For a differentiable covering map $f: \TT^2 \to \TT^2$ and a pair $(x,v) \in T\TT^2$, the number
\[
\Tilde{\mathcal{X}}(x,v) = \limsup\limits_{n \to  \infty} \frac{\log\| D_xf^n(v)\|}{n}
\]
is the Lyapunov exponent of $f$ at $(x,v)$. See \protect\cite{barreira2013introduction} for background in Smooth Ergodic Theory. Due to Oseledet's Theorem \protect\cite{oseledets1968multiplicative} , there is a full area set $M_0$ on $\TT^2$ where the previous limit exists for every $v$, and there exists a measurable bundle $E^-$ defined on $M_0$ such that for $x \in M_0$, $v \neq 0 \in  E^-(x)$:
\[
\mathcal{X}(x,v) := \lim\limits_{n \to \infty}\frac{\log \|D_xf^n(v)\|}{n} = \lim\limits_{n \to \infty}\frac{\log m(D_xf^n)}{n} := \mathcal{X}^-(x),
\]
while for $v \in \rr^2 \setminus E^-(x)$:
\[
\mathcal{X}(x,v) = \lim\limits_{n \to \infty}\frac{\log \|D_xf^n\|}{n} := \mathcal{X}^+(x),
\]
Moreover, if $\mu$ denotes the Lebesgue (Haar) measure on $\TT^2$, then: 
\begin{equation}
    \int(\mathcal{X}^+(x) + \mathcal{X}^-(x)) d\mu(x) = \int \log |\det D_xf| d\mu(x) > 0,
\end{equation}
so $\mathcal{X}^+(x)>0$ almost everywhere. At last, we say that $f$ is non-uniformly hyperbolic (NUH) if $\mathcal{X}^-(x) < 0 < \mathcal{X}^+(x)$ almost everywhere.

Non uniformly hyperbolic systems provide a generalization of the classical Anosov surface maps \protect\cite{Anosov1969GeodesicFO}. Here, we will only be concerned with the non-invertible case in an attempt to aid the understanding of their statistical properties, which is still  under development. For the general ergodic theory of endomorphisms, the reader is directed to \protect\cite{Qian2009smooth}.

Any map $f: \TT^2 \to \TT^2$ is homotopic to a linear endomorphism $E: \TT^2 \to \TT^2$, induced by an integer matrix that we denote  by the same letter. In \protect\cite{NUH}, it is established the existence of a $\mathcal{C}^1$ open set of non-uniformly hyperbolic systems that intersects every homotopy class that does not contain a homothety, provided that the degree is not too small. The authors then conjecture that the same is true for homotheties. In this article, we prove this conjecture, provided that the degree is at least $5^2$. There are other low topological degree cases not covered by Andersson, Carrasco and Saghin, which we also address here.

Let $\text{End}^r_{\mu}(\TT^2)$ be the set of $\mathcal{C}^r$ local diffeomorphisms of $\TT^2$ preserving the Lebesgue measure $\mu$, that are not invertible. For $f \in \text{End}^r_{\mu}(\TT^2)$, $(x,v) \in T^1\TT^2$ define:
\[
I(x,v;f^n) = \sum\limits_{y \in  f^{-n}(x)}\frac{\log \| (D_yf^n)^{-1}v \|}{\det (D_yf^n)},
\]
and
\[
C_{\mathcal{X}}(f) = \sum\limits_{n \in \nn} \frac{1}{n} \inf\limits_{(x,v) \in T^1\TT^2} I(x,v;f^n).
\]
Define the set 
\[
\mathcal{U}:= \{f \in \text{End}^r_{\mu}(\TT^2) \ : \  C_{\mathcal{X}}(f) > 0\},
\]
which is open in the $\mathcal{C}^1$-topology. On Subsection 2.3 of the main reference \protect\cite{NUH}, it is proved:

\begin{xthm}\label{lema funções em U são NUH}
If $f \in \mathcal{U}$, then $f$ is non-uniformly hyperbolic.
\end{xthm}

Our main results are:

\begin{thm}\label{teo.A para homotetias}
For $E = k\cdot Id \in M_{2\times 2}(\zz)$, with $|k| \ge 5$, the intersection $[E] \cap \mathcal{U}$ is non-empty and in fact contains maps that are real analytically homotopic to E.
\end{thm}       

\begin{thm}\label{teo. A  para nao homotetias}
    If $E \in M_{2\times 2}(\zz)$ is not a homothety and $det(E) > 4$, the intersection $[E] \cap \mathcal{U}$ is non-empty and in fact contains maps that are real analytically homotopic to E.
\end{thm}

Our Theorem B is equivalent to the Theorem A of \protect\cite{NUH} but includes three cases which are not proved there. The main difficulty for our results is that, in the case of a homothety, the induced projective action is trivial; non-triviality of this projective action is a central piece in the method of Andersson et al.

Finally, by inspection on the proofs of Theorems B and C of \protect\cite{NUH}, we can see that it works for all cases included here. Hence, defining:
\[
C_{\det}(f) := \sup\limits_{n \in \nn} \frac{1}{n} \inf\limits_{x \in \TT^2} \log (\det(D_xf^n)) > 0,
\]
and the open set:
\[
\mathcal{U}_1 := \left\{f \in \text{End}^r_{\mu}(\TT^2) \ : \  C_{\mathcal{X}}(f) > -\frac{1}{2} C_{\det}(f)  \right\},
\]
we have from Theorems \ref{teo.A para homotetias} and \ref{teo. A  para nao homotetias} that if a linear endomorphism $E$ satisfies the conditions of either of the Theorems, then $[E] \cap \mathcal{U}_1 \neq \emptyset$. Therefore, by Theorem $B$ of \protect\cite{NUH}, we have conituity of the maps $\mathcal{U}_1 \ni f \mapsto \int_{\TT^2} \mathcal{X}^{\pm}(f) d\mu$ in the $\mathcal{C}^1$ topology. 

From Theorem C of \protect\cite{NUH}, we conclude that for any linear endomorphism E as in Theorem \ref{teo.A para homotetias} or \ref{teo. A  para nao homotetias}. If $\pm 1$ is not an eigenvalue of $E$, then $[E] \cap \mathcal{U}$ contains stably ergodic endomorphisms. In fact, it contains stably Bernoulli endomorphisms and, in particular, maps that are mixing of all orders.

\subsection*{Acknowledgements}

The results presented here were conjectured by Martin Andersson, Pablo D. Carrasco and Radu Saghin in \protect\cite{NUH}, I thank Pablo D. Carrasco, who is also my MSc advisor, for the suggestion of the problem and for the hours of conversations on the subject that were crucial to this article. 

This work has been supported by the Brazillian research agencies CAPES and CNPq.
\section{Preliminary}

In order to prove Theorems \ref{teo.A para homotetias} and \ref{teo. A  para nao homotetias}, we require a result on the computation of the numbers $I(x,v;f^n)$ which the proof can be found in \protect\cite{NUH}:

\begin{prop}\label{soma I(x,v;f^n)}
    For any $n \in \nn$, it holds:
    \begin{equation}
        I(x,v;f^n) = \sum\limits_{i=0}^{n-1}\sum\limits_{y \in f^{-i}(x)}\frac{I(y, F_y^{-i}v;f)}{\det (D_yf^i)},
    \end{equation}
    where $F_y^{-i}v = \frac{(D_yf^i)^{-1}v}{\|(D_yf^i)^{-i}v\|}$.
\end{prop}

\subsection{Shears}

For fixed points $z_1, z_2, z_3, z_4 \in \TT^1$, in this order, take the closed intervals $I_1 = [z_1,z_2]$, $I_3 = [z_3,z_4]$, and the open intervals $I_2 = (z_2,z_3)$ and $I_4 = (z_4,z_1)$.

\begin{defn}\label{defn regiões boas e crit} We define the horizontal and vertical critical regions in $\TT^2$ as $\mathcal{C}_h = (I_1 \cup I_3) \times \TT^1$, $\mathcal{C}_v = \TT^1 \times (I_1 \cup I_3)$ and its complements $\mathcal{G}_h = \TT^2 \setminus \mathcal{C}_h$ , $\mathcal{G}_v = \TT^2 \setminus \mathcal{C}_v$ are respectively the horizontal and vertical good region.

We then divide the good regions into $\mathcal{G}_h^+ = I_4 \times  \TT^1$, $\mathcal{G}_h^- = I_2 \times  \TT^1$, $\mathcal{G}_v^+ = \TT^1 \times I_4$ and $\mathcal{G}_v^- = \TT^1 \times I_2$.
\end{defn}

For fixed numbers $0<a<b$, we take $s: \TT^1 \to \rr$ as an analytic map satisfying the following conditions:
\begin{enumerate}
    \item If $z \in I_4$, then $a<s'(z)<b$;
    \item If $z \in I_2$, then $-b<s'(z)<-a$;
    \item If $z \in I_1 \cup I_3$, then $|s'(z)|<b$.
\end{enumerate}

Consider the two families of conservative diffeomorphisms of the torus given by:
\[
h_t(x_1,x_2) = (x_1,x_2 + t s(x_1)), \ \ v_r(x_1,x_2) = (x_1 + r s(x_2),x_2), \ \ \ t,r \in \rr.
\]
Note that:
\[
D_{(x_1,x_2)}h_t = \left( \begin{matrix}
1 & 0\\
ts'(x_1) & 1
\end{matrix}\right), \ \ \ 
D_{(x_1,x_2)}v_r = \left( \begin{matrix}
1 & rs'(x_2)\\
0 & 1
\end{matrix}\right).
\]

In order to simplify the computations we will consider the maximum norm on $T\TT^2$ as $\|(u_1,u_2)\| = \text{max}\{|u_1|,|u_2|\}$, and all the computations from now on are performed using this norm. This way, we get, for every $x \in \TT^2$:

\begin{equation*}
    \|D_xh_t\| < bt+1, \ \ \text{and} \ \ \|D_xv_r\| < bt+1.
\end{equation*}

\begin{defn} \label{defn cones vert e hor}
Given $\alpha>0$, the corresponding horizontal cone is $\Delta_{\alpha}^h=\{(u_1,u_2) \in \rr^2 \ : \ |u_2| \le \alpha|u_1|\}$, while the corresponding vertical cone is its complement $\Delta_{\alpha}^v = \rr^2 \setminus \Delta_{\alpha}^h$,
\end{defn}

\begin{lema}\label{lema derivadas e cones}
For $\alpha > 1$, let $\Delta_{\alpha}^h$ and $\Delta_{\alpha}^v$ be the corresponding horizontal and vertical cones. Then, for every $t,r > \frac{2\alpha}{a}$, and, for every unit vector $u \in T_x\TT^2$, the following holds:
\begin{enumerate}
    \item If $u \in \Delta_{\alpha}^v$, and:
    \begin{enumerate}
        \item $x \in \mathcal{G}_v$, then
        \begin{itemize}
            \item $(D_xv_r)^{-1}u \in \Delta_{\alpha}^h$ \ ($D_xv_r^{-1} \Delta_{\alpha}^v \subset \Delta_{\alpha}^h$);
            \item $\|(D_xv_r)^{-1}u\|>\frac{ar-\alpha}{\alpha} = r\frac{a-\frac{\alpha}{r}}{\alpha}$;
        \end{itemize}
        \item $x \in \mathcal{C}_v$, then $\|(D_xv_r)^{-1}u\|>\frac{1}{\alpha}$.
    \end{enumerate}
    \item If $u = \pm (1,u_2) \in \Delta_{\alpha}^h$, then:
    \begin{enumerate}
        \item either for every $x \in \mathcal{G}_v^+$ ( if $u_2\le 0$) or for every $x \in \mathcal{G}_v^-$ (if $u_2 \ge 0$) it holds:
        \begin{itemize}
            \item $(D_xv_r)^{-1}u \in \Delta_{\alpha}^h$;
            \item $\|(D_xv_r)^{-1}u\| >1$;
        \end{itemize}
        \item for all other $x$, we have $\|(D_xv_r)^{-1}u\|> \frac{1}{br+1}$.
    \end{enumerate}
    \item If $u \in \Delta_{\alpha}^h$, and:
    \begin{enumerate}
        \item $x \in \mathcal{G}_h$, then
        \begin{itemize}
            \item $(D_xh_t)^{-1}u \in \Delta_{\alpha}^v$ \ ($D_xh_t^{-1} \Delta_{\alpha}^h \subset \Delta_{\alpha}^v$);
            \item $\|(D_xh_t)^{-1}u\|>\frac{at-\alpha}{\alpha} = t\frac{a-\frac{\alpha}{t}}{\alpha}$;
        \end{itemize}
        \item $x \in \mathcal{C}_h$, then $\|(D_xh_t)^{-1}u\|>\frac{1}{\alpha}$.
    \end{enumerate}
    \item If $u = \pm (u_1,1) \in \Delta_{\alpha}^v$, then:
    \begin{enumerate}
        \item either for every $x \in \mathcal{G}_h^+$ ( if $u_1\le 0$) or for every $x \in \mathcal{G}_h^-$ (if $u_1 \ge 0$) it holds:
        \begin{itemize}
            \item $(D_xh_t)^{-1}u \in \Delta_{\alpha}^v$;
            \item $\|(D_xh_t)^{-1}u\| >1$;
        \end{itemize}
        \item for all other $x$, we have $\|(D_xh_t)^{-1}u\|> \frac{1}{bt+1}$.
    \end{enumerate}
\end{enumerate}
\end{lema}

\begin{proof}
    We prove items 1 and 2, the case for $h_t$ is analogous. Let $x = (x_1,x_2) \in \mathcal{G}_v$, and $u^{\pm} = (1,\pm \alpha)$ then:
    \[
    (D_xv_r)^{-1}u^{\pm} = \begin{pmatrix}
        1 & -rs'(x_2)\\
        0 & 1
    \end{pmatrix}
    \begin{pmatrix}
        1 \\ \pm \alpha
    \end{pmatrix}
    = \begin{pmatrix}
        1 \mp rs'(x_2)\alpha \\
        \pm \alpha
    \end{pmatrix},
    \]
    also since $x \in \mathcal{G}_v$, $a<|s'(x_2)|<b$, we also have $\alpha>1$ and $r> \frac{2\alpha}{a}$, hence:
    \[
    |1\mp rs'(x_2)\alpha| \ge r\alpha a-1 > 2 \alpha^2 -1 > \alpha>1,
    \]
    which shows that $(D_xv_r)^{-1} \Delta_{\alpha}^v \subset \Delta_{\alpha}^h$. Also, $\|(D_xv_r)^{-1}u\| = |1 \mp rs'(x_2)\alpha|>ra\alpha -1$. Now, noticing that the minimal expansion of vectors in $\Delta_{\alpha}^v$ occurs on either of $(1,\pm \alpha)$, we have for every unit vector $u \in \Delta_{\alpha}^v$:
    \[
    \|(D_xv_r)^{-1}u\| \ge \frac{\|(D_xv_r)^{-1}(1,\pm \alpha)\|}{\|(1,\pm \alpha)\|} > \frac{r\alpha-1}{\alpha}.
    \]

    For part 2 (a), we have for x $\in \mathcal{G}_v^+$ $s'(x_2)>a>0$, and for $x \in  \mathcal{G}_v^-$, $s'(x_2)<-a<0$, thus, by simple calculations analogous to the last one, we get the results. Finally, for (b) we just use $m((D_xv_r)^{-1}) = \frac{1}{\|D_xv_r\|} > \frac{1}{br+1}$ for every $x \in \TT^2$.
\end{proof}

\section{Endomorphisms and Shears: Proof of Theorem \ref{teo.A para homotetias}} \label{seção homotetias}

Fix $E = k \cdot Id$, for some $k \in \nn$ (we shall make the entire argument on $k \in \nn$ for the sake of simplicity of notation, we emphasize that the entire argument works for $k \in \zz$ by replacing $k$ for $|k|$ when necessary). Fix a $\delta <\frac{1}{4k}$ and define the critical and good regions as in Def. \ref{defn regiões boas e crit} for points $z_1, z_2, z_3, z_4 \in \TT^1$ such that:
\begin{itemize}
    \item $I_1=[z_1,z_2]$ and $I_3 = [z_3,z_4]$ have size $2\delta$;
    \item The translation of $I_1$ by a multiple of $\frac{1}{k}$ does not intersect $I_3$.
    \item $I_2 = (z_2,z_3)$ and $I_4=(z_4,z_1)$ have size strictly larger than $\frac{1}{k}\floork$, where $[p]$ denotes the floor of $p$.
\end{itemize}

It is obtained directly from the definitions that:

\begin{prop} \label{prop. pre-imagens de E}
For every $x = (x_1,x_2) \in \TT^2$,  $E^{-1}(x)$ has $k^2$ points given by:
\[
E^{-1}(x_1,x_2) = \left\{\left(\frac{x_1+i}{k},\frac{x_2+j}{k}\right)\ : \ i,j = 0,\cdots,k-1\right\}.
\]
At least $k \floork$ are inside each of $\mathcal{G}_v^+$, $\mathcal{G}_v^-$, $\mathcal{G}_h^+$ and $\mathcal{G}_h^-$, and at most $k$ of them are inside each of $\mathcal{C}_v$, $\mathcal{C}_h$.
\end{prop}

From now on, in this section, we fix any $\alpha>1$ and the corresponding cones as in Def. \ref{defn cones vert e hor}. We consider the analytic maps:
\[
f_{(t,r)} = E \circ v_r \circ h_t,
\]
which we shall denote only by $f=f_{(t,r)}$. Clearly $f$ is an area preserving endomorphism isotopic to E. We observe that, given $x \in  \TT^2$ and $y \in f^{-1}(x)$, we have:
\[
(D_yf)^{-1} = (D_yh_t)^{-1}(D_{h_t(y)}v_r)^{-1}E^{-1}.
\]

The goal is for $(D_{h_t(y)}v_r)^{-1}$ to take vectors in the vertical cone and expand them in the horizontal direction and then $(D_yh_t)^{-1}$ takes its images and expands them in the vertical direction, resulting in $(D_yf)^{-1}$ expanding in the vertical direction for most points in $f^{-1}(x)$. Thus, in order to keep track of this derivative, we must localize the points $y \in f^{-1}(x)$ in regard to which of $\mathcal{G}_h$ or $\mathcal{C}_h$ they belong, and $\{h_t(y) : y \in f^{-1}(x)\} = (E\circ v_r)^{-1}(x)$ regarding which of $\mathcal{G}_v$ or $\mathcal{C}_v$ they belong.

\begin{lema} \label{lema pre-imagens de f}
For every $x \in  \TT^2$, we have:
\begin{enumerate}
    \item $(v_r\circ E)^{-1}(x)$ has $k^2$ points of which at least $k \floork$ of them are in each one of $\mathcal{G}_v^+$ and $\mathcal{G}_v^-$ and at most $k$ of them are in $\mathcal{C}_v$;
    \item $f^{-1}(x)$ has $k^2$ points of which at least $k \floork$ of them are in each one of $\mathcal{G}_h^+$ and $\mathcal{G}_h^-$ and at most $k$ of them are in $\mathcal{C}_h$.
\end{enumerate}
\end{lema}
\begin{proof}
\begin{enumerate}
    \item It is a direct consequence of Prop. \ref{prop. pre-imagens de E} along with the fact that the regions $\mathcal{G}_v^+$, $\mathcal{G}_v^-$ and $\mathcal{C}_v$ are invariant under $v_r$.
    \item Notice that in each row of pre-images by E of a point $x = (x_1,x_2)$ given by $\left\{\left(\frac{x_1+i}{k},\frac{x_2+j_0}{k}\right) : i = 0,\cdots,k-1\right\}$ for a fixed $j_0 \in \{0,\cdots, k-1\}$, $v_r^{-1}$ is a rotation by $-rs\left(\frac{x_2+j_0}{k}\right)$ in the circle $\TT^1 \times \left\{\frac{x_2+j_0}{k}\right\}$. Hence, at least $\floork$ of the $k$ points of this row are inside each one of $\mathcal{G}_h^+$ and $\mathcal{G}_h^-$, and at most 1 is in $\mathcal{C}_h$.
    
    As this is also true for all the $k$ rows of pre-images by E, we get at least $k \floork$ pre-images by $E \circ v_r$ are inside each one of $\mathcal{G}_h^+$ and $\mathcal{G}_h^-$, and at most $k$ pre-images by $E \circ v_r$ are inside $\mathcal{C}_h$. Finally, since these sets are invariant under $h_t$, we get the desired result.
\end{enumerate}
\end{proof}

\begin{rmk} \label{obs. pre-imagens de f}
Even knowing which regions is a point $y \in (E \circ v_r)^{-1}(x)$, we cannot determine the region which $h_t^{-1}(y)$ is inside, as $t$ is varying. That is, there may be points $y \in f^{-1}(x)$ that are in $\mathcal{G}_h$ such that $h_t(y) \in \mathcal{C}_v$ and vice-versa.
\end{rmk}

\begin{defn} In order to keep track of the vectors, define:
\begin{itemize}\label{defn sinal dos vetores}
    \item For $u=(u_1,u_2) \in \rr^2$ with $u_2 \neq 0$:
    \[
    *(u) = \left\{\begin{array}{l}     -\text{sgn}\left(\frac{u_1}{u_2}\right), \ \ \text{if} \  u_1 \neq 0,\\
     -\text{sgn}(u_2), \ \ \ \ \text{if} \ u_1=0.
    \end{array}
    \right.
    \]
    
    Notice that $*(u) = \ *(E^{-1}u)$, for every $u \in \rr^2$.

    \item For $x \in \TT^2$, $y \in f^{-1}(x)$ and $u \in \rr^2$, let $(w_1,w_2) = (D_{h_t(y)}v_r)^{-1}E^{-1}u$:
    \[
    *_y(u) =\left\{\begin{array}{l} -\text{sgn}\left(\frac{w_1}{w_2}\right), \ \ \text{if} \ w_1,w_2\neq 0, \\
    -\text{sgn}(w_2), \ \ \ \ \text{if} \ w_2\neq 0, \ w_1=0,\\
    -\text{sgn}(w_1), \ \ \ \ \text{if} \ w_1 \neq 0, \ w_2 = 0.
    \end{array}
    \right.
    \]
\end{itemize}
\end{defn}

In view of item 4 of Lemma \ref{lema derivadas e cones}, even though $(D_{h_t(y)}v_r)^{-1}$ may not send a vector \linebreak $u \in \Delta_{\alpha}^v$ to the horizontal cone if $h_t(y) \in \mathcal{C}_v$, we can still end up having expansion in the vertical direction, depending on whether $y \in \mathcal{G}_h^{*_y(u)}$ or not. In this regard, from Lemma \ref{lema pre-imagens de f}, there are $k$ points $y \in f^{-1}(x)$ such that $h_t(y)$ are in $\mathcal{C}_v$, and these points ($h_t(y)$) are all in the same circle $\TT^1 \times \left\{\frac{x_2+j_0}{k}\right\}$, hence the derivative $(D_{h_t(y)}v_r)^{-1}$ is the same for those points. We get:

\begin{prop}\label{prop. sinal constante}
For every $u \in \rr^2$, $x \in \TT^2$, then the sign $*_y(u) = \text{sg}\left(\frac{w_1}{w_2}\right)$ is the same for all points $y \in f^{-1}(x)$ such that $h_t(y) \in \mathcal{C}_v$, where $*_y(u)$ is as in Definition \ref{defn sinal dos vetores}.
\end{prop}

\begin{defn} For a fixed $x \in \TT^2$ and:
\begin{itemize}
    \item $u \in \Delta_{\alpha}^v$, define: 
    $\left\{\begin{array}{l}
    A = \{y \in f^{-1}(x) : y \in \mathcal{G}_h, h_t(y) \in \mathcal{G}_v\}.\\ 
    B = \{y \in f^{-1}(x) : y \in \mathcal{G}_h^{*_y(u)}, h_t(y) \in \mathcal{C}_v\},\\
    \mathcal{V}_v = A \cup B,\\
    \mathcal{V}_h = f^{-1}(x) \setminus \mathcal{V}_v. \end{array}\right.$
    
    \item  $u \in \Delta_{\alpha}^h$, define:
    $\left\{\begin{array}{l}
    C = \{y \in f^{-1}(x) : y \in \mathcal{G}_h, h_t(y) \in \mathcal{G}_v^{*(u)}\}.\\
    D = \{y \in f^{-1}(x) : y \in \mathcal{G}_h^{*_y(u)}, h_t(y) \in \mathcal{C}_v \cup \mathcal{G}_v^{-*(u)}\},\\
    \mathcal{H}_v = C \cup D,\\
    \mathcal{H}_h = f^{-1}(x) \setminus \mathcal{H}_v. \end{array}\right.$
\end{itemize}
\end{defn}

A direct consequence of Lemma \ref{lema pre-imagens de f} and Prop. \ref{prop. sinal constante}, having Remark. \ref{obs. pre-imagens de f} in mind, is the following:

\begin{lema}\label{lema quant. pre-im. em cada conj.}
For a fixed $(x,u) \in T\TT^2$, $f^{-1}(x)$ has $k^2$ points, of which:
\begin{enumerate}
    \item For $u \in \Delta_{\alpha}^v$, at most $2k-1-\floork$ of them are in $\mathcal{V}_h$ and at least $(k-1)^2+\floork$ are inside $\mathcal{V}_v$, because:
    \begin{itemize}
        \item At least $(k-1)^2$ are in A and,
        \item at least $\floork$ are in B.
    \end{itemize}
    \item For $u \in \Delta_{\alpha}^h$, at most $k^2 - \floork \left(k+\floork \right)$ are in $\mathcal{H}_h$ and at least $\floork\left(k+\floork\right)$ are in $\mathcal{H}_v$, because:
    \begin{itemize}
        \item At least $(k-1)\floork$ are in C and,
        \item at least $\floork\left(1+\floork\right)$ are in D.
    \end{itemize}
\end{enumerate}
\end{lema}

Knowing that for every unit vector $u \in \rr^2$ we have $\|E^{-1}u\| = \frac 1k$ (maximum norm), from Lemma \ref{lema derivadas e cones} we get:

\begin{lema}\label{lema derivadas e cones de f}
For $t,r>\frac{2\alpha}{a}$ and for fixed $x \in \TT^2$, it holds:
\begin{enumerate}
    \item If $u \in \Delta_{\alpha}^v$, then for all $y \in \mathcal{V}_v$ we have $(D_yf)^{-1}u \in \Delta_{\alpha}^v$;
    \item If $u \in \Delta_{\alpha}^v$ is a unit vector, then:
    \[
    \|(D_yf)^{-1}u\|>\left\{\begin{array}{l}
         \left(\frac{a-\frac{\alpha}{t}}{\alpha}\right)\left(\frac{a-\frac{\alpha}{r}}{\alpha}\right)\frac{tr}{k}, \  y \in A,  \\
         \frac{1}{\alpha k}, \hspace{2.35cm}  y \in B, \\
         \frac{1}{(bt+1)\alpha k}, \hspace{1.6cm}  y \in \mathcal{V}_h;
         \end{array}\right.
    \]
    \item If $u \in \Delta_{\alpha}^h$, then for all $y \in \mathcal{H}_v$ we have $(D_yf)^{-1}u \in \Delta_{\alpha}^v$;
    \item If $u \in \Delta_{\alpha}^h$ is a unit vector, then:
    \[
    \|(D_yf)^{-1}u\| > \left\{\begin{array}{l}
         \left(\frac{a-\frac{\alpha}{t}}{\alpha}\right)\frac{t}{k}, \hspace{0.4cm} y \in C, \\
         \frac{1}{(br+1)k}, \hspace{0.85cm}  y \in D, \\
         \frac{1}{(bt+1)(br+1)k}, \ y \in \mathcal{H}_h.
    \end{array}\right.
    \]
\end{enumerate}
\end{lema}

\subsection{Non-uniform hyperbolicity}\label{subseção NUH homotetias}

For $(x,u) \in T\TT^2$ with $u \neq 0$ and for $n \in \nn$ denote by
\[
Df^{-n}(x,u) = \{(y,w) \in T\TT^2 : f^n(y) = x, D_yf^nw = u\}.
\]
For any non-zero tangent vector $(x,u)$ and $n \ge 0$, define:
\begin{align*}
    &\mathcal{G}_n = \{(z,w) \in Df^{-n}(x,u) : w \in \Delta_{\alpha}^v\},\\
    &\mathcal{B}_n = Df^{-n}(x,u)\setminus \mathcal{G}_n,\\
    & g_n = \#\mathcal{G}_n,\\
    & b_n = \#\mathcal{B}_n = k^{2n} - g_n.
\end{align*}

From Lemmas \ref{lema quant. pre-im. em cada conj.}, \ref{lema derivadas e cones de f} one deduces:

\begin{lema}Let $(x,u) \in T\TT^2$.
\begin{enumerate}
    \item If $u \in \Delta_{\alpha}^v$, then at least $(k-1)^2+\floork$ of its pre-images under $Df$ are also in $\Delta_{\alpha}^v$;
    \item If $u \in  \Delta_{\alpha}^h$, then at least $\floork\left(k+\floork\right)$ of its pre-images under $Df$ are in $\Delta_{\alpha}^v$.
\end{enumerate}
\end{lema}

By the lemma above, we get:
\begin{align*}
    g_{n+1} & \ge \left((k-1)^2+\floork\right)g_n + \floork\left(k+\floork\right)b_n\\
    & = \left((k-1)^2-\floork\left(k-1+\floork\right)\right) g_n + \floork\left(k+\floork\right)k^{2n},
\end{align*}
hence:
\begin{align*}
\frac{g_{n+1}}{k^{2(n+1)}} \ge \frac{1}{k^2}\left((k-1)^2-\floork\left(k-1+\floork\right)\right)&\frac{g_n}{k^{2n}}\\ &+\frac{1}{k^2}\floork\left(k+\floork\right).
\end{align*}
Denoting by $a_n = \frac{g_n}{k^{2n}}$ and
\begin{align*}
    &c = \frac{1}{k^2}\left((k-1)^2-\floork\left(k-1+\floork\right)\right),\\
    &e = \frac{1}{k^2}\floork\left(k+\floork\right),
\end{align*}
the inequality above becomes:
\[
a_{n+1} \ge c \cdot a_n + e.
\]

\begin{lema} \label{lema limitantes prop. de pre-im.}
For every $(x,u) \in T\TT^2$, $u \neq 0$, and $n \ge 0$ it holds:
\begin{align*}
a_n &\ge \frac{e}{1-c}(1-c^n)\\
&= \frac{\floork\left(k+\floork\right)}{2k-1+\floork \left(k-1+\floork\right)}(1-c^n)
\end{align*}
In particular,
\[
\liminf a_n \ge \frac{\floork\left(k+\floork\right)}{2k-1+\floork \left(k-1+\floork\right)}:=L(k),
\]
uniformly in $(x,u) \in \TT^2$.
\end{lema}

From now on we shall denote by $L(k) = \frac{\floork\left(k+\floork\right)}{2k-1+\floork \left(k-1+\floork\right)}$. As another direct consequence of Lemmas  \ref{lema quant. pre-im. em cada conj.} and \ref{lema derivadas e cones de f} we  have the following:

\begin{lema} \label{lema limitantes I(x,u;f)}
If $r,t > \frac{2\alpha}{a}$, then for all $(x,u) \in T\TT^2$ we have:
\begin{enumerate}
    \item If $u \in  \Delta_{\alpha}^v$, then:
    \begin{align*}
        I(x,u;f) \ge &\frac{(k-1)^2}{k^2}\log r + \left(\frac{k^2-4k+2+\floork}{k^2} \right)\log t\\
        &+ \log \left(\frac{1}{\alpha k} \left(\left(a-\frac{\alpha}{t} \right)\left(a - \frac{\alpha}{r} \right) \right)^{\frac{(k-1)^2}{k^2}} \left(b+\frac{1}{t} \right)^{-\frac{1}{k^2}\left(2k-1-\floork \right)}\right).
    \end{align*}
    \item If $u \in \Delta_{\alpha}^h$, then:
    \begin{align*}
        I(x,u;f) \ge &-\left(\frac{k^2-(k-1)\floork}{k^2} \right)\log r -\left(\frac{k^2-\floork\left(2k-1+\floork \right)}{k^2} \right)\log t \\ 
        &+ \log \left(\frac{1}{k} \left(\frac{1}{\alpha}\left(a - \frac{\alpha}{t} \right) \right)^{\frac{k-1}{k^2}\floork -1} \left(b+\frac{1}{t} \right)^{\frac{1}{k^2}\floork\left(k+\floork \right)-1 } \right).
    \end{align*}
\end{enumerate}
\end{lema}

Now, to calculate $\mathcal{C}_{\mathcal{X}}(f)$, we use Prop. \ref{soma I(x,v;f^n)} to compute:
\[
I(x,u;f^n) = \sum\limits_{i=0}^{n-1} \sum\limits_{y \in f^{-i}(x)} \frac{I(y,(D_yf^{i})^{-1}u;f)}{k^{2i}} := \sum\limits_{i=0}^{n-1} J_i,
\]
and, if $t,r > \frac{2\alpha}{a}$, for each $i$ we obtain:
\begin{align*}
    J_i &= \frac{1}{k^{2i}}\sum\limits_{y \in f^{-1}(x)} I(y, (D_yf^i)^{-1}u;f) = \frac{1}{k^{2i}} \sum\limits_{(y,w) \in \mathcal{G}_i} I(y,w;f) + \frac{1}{k^{2i}} \sum\limits_{(y,w) \in \mathcal{B}_i} I(y,w;f) \\
    &\ge a_i V(t,r,k) + (1- a_i) H(t,r,k),
\end{align*}
where V and H are the right side of the inequalities obtained in Lemma \ref{lema limitantes I(x,u;f)} for $u \in \Delta_{\alpha}^v$ and $u \in \Delta_{\alpha}^h$ respectively. It follows from Lemma \ref{lema limitantes prop. de pre-im.}, with $L(k)$ as above and $c_k = \floork$, to simplify the notation, that:
\begin{align*}
    \lim\limits_{i \to \infty}J_i &\ge L(k) V(t,r,k) + (1-L(k))H(t,r,k) \nonumber \\ 
    &= C(t,r,k) +\frac{1}{k^2}\left(L(k)\left((k-1)\left(2k-c_k \right) +1 \right) - \left(k^2 - (k-1)c_k \right) \right) \log r \ + \nonumber \\
    & \ \ \ \ \frac{1}{k^2} \left(L(k)\left(2(k-1)^2 - c_k\left( 2(k-1)+c_k \right)\right) -\left(k^2 - c_k \left( 2k-1+ c_k \right) \right) \right) \log t&,
\end{align*}
where
\[
C(t,r,k) = L(k) C_1(t,r,k) + (1-L(k))C_2(t,r,k),
\]
with
\begin{align*}
    C_1(t,r,k) &= \log \left(\frac{1}{\alpha k} \left(\left(a-\frac{\alpha}{t} \right)\left(a - \frac{\alpha}{r} \right) \right)^{\frac{(k-1)^2}{k^2}} \left(b+\frac{1}{t} \right)^{-\frac{1}{k^2}\left(2k-1-\floork \right)}\right) \\
    C_2(t,r,k) &= \log \left(\frac{1}{k} \left(\frac{1}{\alpha}\left(a - \frac{\alpha}{t} \right) \right)^{\frac{k-1}{k^2}\floork -1} \left(b+\frac{1}{t} \right)^{\frac{1}{k^2}\floork\left(k+\floork \right)-1 } \right),
\end{align*}
as in Lemma \ref{lema limitantes I(x,u;f)}. From this, we get that for any $k$, $C(t,r,k)$ is growing as $t$ and $r$ grow, then for $t,r> \frac{2\alpha}{a}$, $C(t,r,k) > C$ is uniformly bounded from below by some constant $C$.

Now, in order to get $\lim\limits_{i\to \infty}J_i > 0$, we can either make $t$ or $r$ large, depending on whether the constant (which depends on $k$) multiplying $\log t$ or $\log r$ is positive or negative. However, for both of them, we only get positivity of the constant if $k \ge 5$.

Thus, for $k \ge 5$, since all the bounds above are uniform for all non-zero tangent vectors $(x,u)$, we obtain that for $t$ (or $r$) sufficiently large, for all $i$ greater than some $i_0$, and for all nonzero tangent vectors $(x,u)$, $J_i(x,u) > N > 0$ for some constant $N$. Hence, there exists some $n_0$ such that
\[
\frac{1}{n_0}I(x,u;f^{n_0}) = \frac{1}{n_0} \sum\limits_{i=0}^{n_0-1}J_i(x,u) > \frac{N}{2} >0,
\]
for all nonzero tangent vectors $(x,u)$. Therefore, $\mathcal{C}_{\mathcal{X}}(f)>0$ which by Theorem \ref{lema funções em U são NUH} concludes the proof of Theorem \ref{teo.A para homotetias}.

We finish this section by including some examples for a better visualization that for a fixed $k \in  \nn$, the bounds obtained in this section are quite simple. For that, we fix $k=5$, we get $L(5) = \frac{2}{3}$, the limitations of our last calculations become:
\[
    \lim\limits_{i\to \infty} J_i \ge C(t,r,5) + 5 \log r+5 \log t,
\]
with 
\[
C(t,r,5) = \log \left(\frac{1}{5}\frac{\alpha^{\frac{17}{25}}}{a^{2/3}} \left(a-\frac{\alpha}{t} \right)^{\frac{1}{5}} \left(a-\frac{\alpha}{r} \right)^\frac{32}{75}\left(b+\frac{1}{t} \right)^{-\frac{18}{25}} \right)
\]

Thus, taking the map $s: \TT^1 \to \rr$ as $s(u) = \sin(2\pi u)$, $\delta = \frac{1}{20}$, $a = 2\pi\sin(\frac{\pi}{10})$, $b = 2\pi$, and $\alpha = 1.1$, we get that for every $t,r \gtrapprox \frac{2a}{\alpha} \approx 1.77$ the number $C(t,r,5) + 5 \log r+5 \log t$ is positive. Thus, the maps $f_(t,r) = E \circ v_r \circ h_t$ satisfy the results of Theorem \ref{teo.A para homotetias}.

\section{Proof of Theorem \ref{teo. A  para nao homotetias}}\label{seção casos grau pequeno}

For $ k\cdot Id \neq  E \in M_{2\times 2}(\zz)$, let $\tau_1(E)$ be the greatest common divisor of the entries of E, $\tau_2(E) = \det(E)/\tau_1(E)$, so that $d = \tau_1 \cdot \tau_2$ coincides with the topological degree of the induced endomorphism $E : \TT^2 \to \TT^2$. 

We want to make a slight change in the argument used in \protect\cite{NUH} so that for every $x \in \TT^2$, $f^{-1}(x)$ has at most one point in the critical zone. This solves the cases where the pair $(\tau_1,\tau_2)$ is $(2,4), (3,3)$ or $(4,4)$. For the remaining four cases $(1,2), (1,3), (1,4)$ and $(2,2)$, even with this improvement in the argument, the proportion we obtain for vectors in the good region (which in these cases is the optimum one for the argument presented here) is still insufficient to obtain expansion in the vertical direction, given the small amount of pre-images.

The numbers $\tau_1, \ \tau_2$ are the elementary divisors of E and, as in Section 2.4 of \protect\cite{NUH}, there exists $P \in GL_2(\zz)$ such that the matrix $G = P^{-1}\cdot E \cdot P$ satisfies:
\[
G^{-1}(\zz) = \left\{ \begin{pmatrix}
    \frac{i}{\tau_2} \\
    \frac{j}{\tau_1}
\end{pmatrix}: i,j \in \zz \right\}
\]

Moreover, as E is not a homothety, by another change of coordinates if necessary we may assume that E does not have $(0,1)$ as an eigenvector.

With this in mind, we assume that $\mathbb{P}E$ does not fix $[(0,1)]$ and that $E^{-1}\zz^2 = \frac{1}{\tau_2}\zz \times \frac{1}{\tau_1} \zz $. So there exists an $\alpha>\tau_2>1$ such that if $\Delta_{\alpha}^h$ and $\Delta_{\alpha}^v$ are the corresponding horizontal and vertical cones as in Def. \ref{defn cones vert e hor}, then $\overline{E^{-1}\Delta_{\alpha}^v} \subset Int(\Delta_{\alpha}^h)$. From now on, we fix such $\alpha > \tau_2$.

Let $L < \max\left\{\frac{1}{4\tau_2}, \frac{\tau_2^{-1}-\alpha^{-1}}{2}\right\}$, choose points $z_1,z_2,z_2,z_4 \in \TT^1$, in this order, such that:
\begin{itemize}
    \item $I_1 = [z_1,z_2]$ and $I_3 = [z_3,z_4]$ have size $L$;
    \item the translation of $I_1$ by a multiple of $1/\tau_2$ does not intersect $I_3$;
    \item $I_2 = (z_2,z_3)$ and $I_4 = (z_4,z_1)$ have size strictly larger than $\frac{1}{\tau_2} \floort$,
\end{itemize}
and define the critical and good regions $\mathcal{C}_h$, $\mathcal{G}_h$ and $\mathcal{G}_h^{\pm}$ as in Def. \ref{defn regiões boas e crit}. As an immediate consequence of the definition we get:

\begin{prop}\label{prop. quant. pontos em cada area boa}
For every $x \in \TT^2$,  $E^{-1}(x)$ has $d$ points of which at least $\frac{1}{\tau_2} \floort$ are inside each of $\mathcal{G}_h^+$ and $\mathcal{G}_h^-$, and at most $\tau_1$ of them are inside of $\mathcal{C}_h$.
\end{prop}

In order to have at most one pre-image of each point in the critical zone of the shear $h_t(x_1,x_2) = (x_1, x_2+ts(x_1)$ defined as before, we define the conservative diffeomorphism of the torus $v(x_1,x_2) = (x_1 + \Tilde{s}(x_2),x_2)$, with $\Tilde{s}: \TT^1 \to \rr$ an analytic map which we shall impose restrictions later. We then study the family:
\[
f_t = E \circ v \circ h_t,
\]
of area preserving endomorphism of the torus isotopic to E. We shall denote $f = f_t$ to simplify the notation.

Given $x \in \TT^2$, the set $f^{-1}(x) = h_t^{-1} \circ v^{-1} \circ E^{-1}(x)$ is composed by d points, and given $y \in f^{-1}(x)$, we have $(D_yf)^{-1} = (D_yh_t)^{-1} \circ (D_{h_t(y)}v)^{-1} \circ E^{-1}$.

In order to define $v$ in a way that only one pre-image of $x$ by $f$ remains in the critical zone, we notice that $E^{-1}(x)$ is composed by $d$ points which, by the change of coordinates made initially, are aligned in a lattice of height $\tau_1$ and length $\tau_2$. We also notice that the map $h_t^{-1}$ keeps the vertical lines invariant. Therefore, the map $v^{-1}$ needs to act in a way that it moves points on a vertical line enough so that only one remains in the critical zone, and, also, it cannot move them so much that we have new points entering the critical zone.

In this way, we took the analytic map $\Tilde{s}: \TT^1 \to \rr$ satisfying:
\begin{enumerate}
    \item If $L$ is the size of the intervals $I_1, I_3$ then $|\Tilde{s}(u)| < \frac{1}{\tau_2}-L$, for all $u \in \TT^1$.
    \item For all $u \in  \TT^1$, we have that $\left|\Tilde{s}\left( u + \frac{j}{\tau_1} \right) \right| > L$ for all $j \in \{0, 1, \cdots, \tau_1-1\}$ except at most one index.
    \item $|\Tilde{s}'(u)| < (2\alpha)^{-1}$, for all $u \in \TT^1$, where $\alpha$ is the size of the cones fixed in the previous subsection.
\end{enumerate}

Notice that conditions 2 and 3 are not mutually exclusives thanks to the conditions for $\alpha$ and $L$ imposed in the previous subsection. Now, conditions 1 and 2 give us:
\begin{lema} \label{lema quatidade pre-imagens nao homotetias}
    For all $x \in \TT^2$, $f^{-1}(x)$ is composed by $d$ points of which at most one is inside $\mathcal{C}_h$. At least $d-1$ of the pre-images are inside $\mathcal{G}$ of which at least $\tau_1 \floort$ are inside each of $\mathcal{G}_h^+$ and $\mathcal{G}_h^-$.
\end{lema}
\begin{proof}
    In the case where $E^{-1}(x)$ has no points in the critical zone, due to condition 1 together with the fact that $h_t$ preserves vertical lines, the map $h_t^{-1} \circ v^{-1}$ does not take any of those points to the critical zone.

    In the case where $E^{-1}(x)$ has a point in the critical zone, it implies that we have exactly $\tau_1$ points there. Due to condition 2, only one of those points is able to remain there, and due to condition 1, none of the other points is getting inside. 

    For the minimum amount of points in each of $\mathcal{G}_h^+$ and $\mathcal{G}_h^-$, we notice that, by Prop. \ref{prop. quant. pontos em cada area boa}, $E^{-1}(x)$ already has at least $\tau_1 \floort$ points inside each one, and, due to condition 1, those points must remain there. 
\end{proof}

At last, condition 3 gives us the next lemma, required for the whole construction to work:

\begin{lema}\label{lema existencia beta cone invariante}
    There exists $\beta> \alpha$ such that for all $y \in \TT^2$, $\overline{(D_yv)^{-1} \circ E^{-1} \Delta_{\beta}^v} \subset \Delta_{\beta}^h$, where $\Delta_{\beta}^v$ and $\Delta_{\beta}^h$ are the corresponding vertical and horizontal cones of size $\beta$ as in Def. \ref{defn cones vert e hor}.
\end{lema}

\begin{proof}
    For $y = (y_1,y_2)$, $D_yv = \begin{pmatrix}
        1 & \Tilde{s}'(y_2) \\
        0 & 1
    \end{pmatrix}$. Then, due to condition 3, for all $\lambda \in \rr$, $D_y v \cdot \lambda e_2 = \lambda (\Tilde{s}'(y_2),1) \in \Delta_{2\alpha}^v$. Since, by the definition of $\alpha$, we have $E^{-1} \cdot \lambda  e_2 \in int(\Delta_{\alpha}^h)$, we conclude that for all $y \in \TT^2$, $\mathbb{P} ((D_yv)^{-1}\circ E^{-1})\cdot [e_2]$ is uniformly away from $[e_2]$, hence there exists such $\beta$ as we wanted.
\end{proof}

\begin{rmk}\label{obs lema também vale}
    Items 3 and 4 of Lemma \ref{lema derivadas e cones} also works in this cases for $\Delta_{\beta}^v$ and $\Delta_{\beta}^h$.
\end{rmk}

We give the correspondent to Lemma \ref{lema derivadas e cones de f} for this case, as a consequence of items 3 and 4 of Lemma \ref{lema derivadas e cones}, Remark \ref{obs lema também vale} and Lemma \ref{lema existencia beta cone invariante} . From now on, we fix $\beta>\alpha$ as in Lemma \ref{lema existencia beta cone invariante} and let:
\[
e_v = \inf\left\{\|(D_xv)^{-1}\circ E^{-1}u\| : (x,u) \in T^1\TT^2, \ u \in \Delta_{\beta}^v \right\},
\]
\[
e_h = \inf\left\{\|(D_xv)^{-1}\circ E^{-1}u\| : (x,u) \in T^1\TT^2, \ u \in \Delta_{\beta}^h \right\}.
\]
\begin{lema}\label{lema  derivadas e cones de f nao homotetias}
    For $t> \frac{2 \beta}{a}$ it holds:
    \begin{enumerate}
        \item if $y \in \mathcal{G}_h$ then $\overline{(D_yf)^{-1}\Delta_{\beta}^v} \subset \Delta_{\beta}^v$, it is strictly invariant.
        \item if $u \in \Delta_{\beta}^v$ is a unit vector, then
        \[
        \|(D_yf)^{-1}u\| > \left\{\begin{array}{l}
        \frac{e_v\left(a- \beta/t) \right)}{\beta}t, \ \  y \in \mathcal{G}_h,\\
        \frac{e_v}{\beta}, \hspace{1.35cm}  y \in \mathcal{C}_h.
        \end{array}\right.
        \]
        \item if $u \in \Delta_{\beta}^h$, and $(D_{h_t(y)}v)^{-1} \circ E^{-1} \cdot u = (w_1,w_2)$ let $*_y(u)$ be as in Def. \ref{defn sinal dos vetores}. Then if $y \in \mathcal{G}_{h}^{*_y(u)}$ we have $(D_yf)^{-1}(u) \in \Delta_{\beta}^v$.
        \item if $u \in \Delta_{\beta}^h$ is a unit vector, then
        \[
        \|(D_yf)^{-1}u\| > \left\{\begin{array}{l}
        e_h, \hspace{0.85cm} y \in  \mathcal{G}_h^{*_y(u)},\\
        \frac{e_h}{b+\frac{1}{t}}t^{-1}, \ y \notin \mathcal{G}_h^{*_y(u)}.
        \end{array}\right.
        \]
    \end{enumerate}
\end{lema}

We notice that, analogously to the homothety case, we have the problem that $*_y(u)$ depends on $y \in f^{-1}(x)$, therefore even though we have at least $\tau_1 \floort$ points in each of $\mathcal{G}_h^{\pm}$, there could be a vector $u \in \rr^2$ such that for all $y \in \mathcal{G}_h^+$, $*_y(u) = -$ and vice-versa. However, we can see that this is not the case:

\begin{prop}\label{prop pre imagens e sinais nao homotetias}
    For every $x \in \TT^2$, $u \in \rr^2$, there are at least $\tau_2 \floort$ points $y \in f^{-1}(x)$ such that $y \in \mathcal{G}_h^{*_y(u)}$, where $*_y(u)$ is as in Def. \ref{defn sinal dos vetores} changing $v_r$ for $v$.
\end{prop}

\begin{proof}
By the same argument used in Prop. \ref{prop. sinal constante}, we can see that $*_y(u)$ is constant for points $y \in f^{-1}(x)$ such that $h_t(y)$ lies in the same horizontal line. There are exactly $\tau_2$ pre-images $y'$ such that $h_t(y)$ and $h_t(y')$ are in the same horizontal line, hence at least $\floort$ of these lies in $\mathcal{G}_h^{*_y(u)}$. As $v^{-1}\circ E^{-1}(x)$ has $\tau_1$ different vertical lines, we get the result. 
\end{proof}

\subsection{Non-uniform hyperbolicity}

We end up having calculations completely mirrored in those made in Subsection \ref{subseção NUH homotetias}, and for that reason we will skip the details. For $(x,u) \in T\TT^2$ with $u \neq 0$ and for $n \in \nn$, we define the sets $Df^{-n}(x,u)$, $\mathcal{G}_n$, $\mathcal{B}_n$, and the numbers $g_n$, $b_n = d^n-g_n$ as before. From Lemmas \ref{lema quatidade pre-imagens nao homotetias},  \ref{lema  derivadas e cones de f nao homotetias} and Prop. \ref{prop pre imagens e sinais nao homotetias} we deduce:

\begin{lema}
    Let $(x,u) \in T\TT^2$.
    \begin{enumerate}
        \item If $u \in \Delta_{\beta}^v$, then at least $d-1$ of its pre-images under $Df$ are also in $\Delta_{\beta}^v$.
        \item If $u \in  \Delta_{\beta}^h$, then at least $\tau_1\floort$ of its pre-images under $Df$ are in $\Delta_{\beta}^h$.
    \end{enumerate}
\end{lema}

For that, we get for all $n \in \nn$:
\[
g_{n+1} \ge \left(d-1-\tau_1\floort \right)g_n +\tau_1\floort d^n,
\]
hence, putting $a_n = \frac{g_n}{d^n}$:
\[
a_{n+1} \ge \left(\frac{d-1}{d} - \frac{1}{\tau_2}\floort \right)a_n + \frac{1}{\tau_2}\floort.
\]
Thus, we get:
\begin{lema}\label{lema an nao homotetia}
    For every $(x,u) \in T\TT^2$, $u \neq 0$, and $n \ge 0$, it holds:
    \[
    \liminf a_n \ge \frac{1}{\tau_2}\floort\frac{d}{1+\tau_1\floort} := L(\tau_1,\tau_2).
    \]
\end{lema}

\begin{rmk}
    This is where we are able to verify that this argument will work for the cases $(\tau_1,\tau_2)$ as $(2,4), (3,3)$ and $(4,4)$, where we have $L(\tau_1,\tau_2)$ as $2/3$, $3/4$ and $4/5$, respectively. And it won't work for the other cases $(1,2), (1,3), (1,4)$ and $(2,2)$ where we will get $L(\tau_1,\tau_2)$ as $0$, $1/2$, $1/2$ and $0$, respectively. As we will see, for the rest of the argument to work, we need this lower bound strictly greater than $1/2$.
\end{rmk}

As another consequence of Lemmas \ref{lema quatidade pre-imagens nao homotetias},  \ref{lema  derivadas e cones de f nao homotetias} and Prop. \ref{prop pre imagens e sinais nao homotetias}, we get:

\begin{lema}\label{lema limitação I(x,u;f) nao homotetias}
    If $t> \frac{2\beta}{a}$, then for all $(x,u) \in T\TT^2$, it holds:
    \begin{enumerate}
        \item If $u \in \Delta_{\beta}^v$, then:
        \[
        I(x,u;f) \ge  \frac{d-1}{d} \log t + \log\left(\frac{e_v}{\beta}\left(a-\frac{\beta}{t} \right)^{\frac{d-1}{d}} \right).
        \]
        \item If $u \in \Delta_{\beta}^h$, then:
        \begin{align*}
          I(x,u;f) \ge -\left(1-\frac{1}{\tau_2}\floort \right) \log t + \log \left(e_h \left(b+\frac{1}{t} \right)^{- \left( 1 - \frac{1}{\tau_2}\floort \right)}\right).  
        \end{align*}
    \end{enumerate}
\end{lema}

Again, by Prop. \ref{soma I(x,v;f^n)}, we have:
\[
I(x,u;f^n) = \sum\limits_{i=0}^{n-1} \sum\limits_{y \in f^{-i}(x)} \frac{I(y,(D_yf^{i})^{-1}u;f)}{k^{2i}} := \sum\limits_{i=0}^{n-1} J_i,
\]
we compute, for $t> \frac{2\beta}{a}$, for all $i \ge 0$:
\begin{align*}
    J_i &= \frac{1}{d}\sum\limits_{(y,w) \in \mathcal{G}_i} I(y,w;f) + \frac{1}{d} \sum\limits_{(y,w) \in \mathcal{B}_i} I(y,w;f)\\
    &\ge a_i V(t,\tau_1,\tau_2) + (1-a_i)H(t,\tau_1,\tau_2),
\end{align*}
where $a_i$ is as in Lemma \ref{lema an nao homotetia}, $V$ and $H$ are the right side of the  inequalities obtained in Lemma \ref{lema limitação I(x,u;f) nao homotetias} for $u \in \Delta_{\beta}^v$ and $u \in \Delta_{\beta}^h$ respectively. It follows:
\begin{align*}
    \lim\limits_{i \to \infty} J_i &\ge L(\tau_1, \tau_2) V(t,\tau_1, \tau_2) + (1-L(\tau_1, \tau_2)) H(t, \tau_1, \tau_2) \\
    &= \frac{\left( \tau_1-\frac{2}{\tau_2} \right)\floort-1}{1+\tau_1\floort} \log t + C(t,\tau_1,\tau_2),
\end{align*}
where:
\begin{align*}
    C(t,\tau_1,\tau_2) = &L(\tau_1,\tau_2) \log \left( \frac{e_v}{\beta} \left(a-\frac{\beta}{t} \right)^{\frac{d-1}{d}} \right) \\ &+ (1-L(\tau_1,\tau_2)) \log \left(e_h \left(b+\frac{1}{t} \right)^{-\left(1-\frac{1}{\tau_2}\floort \right)} \right) > C,
\end{align*}
for all $t > \frac{2\beta}{a}$, that is, $C(t,\tau_1,\tau_2)$ is uniformly bounded from below by some constant C.

Since $d = \tau_1\cdot\tau_2 > 4$, the constant multiplying $\log t$ is positive. Therefore, since all the bounds above are uniform for all non-zero tangent vectors $(x,u)$, as in the homothety case we obtain that for $t$ sufficiently large, for all $n$ greater than some $n_0$, and for all nonzero tangent vectors $(x,u)$:
\[
\frac{1}{n}I(x,u;f^{n}) = \frac{1}{n} \sum\limits_{i=0}^{n-1}J_i(x,u)>0,
\]
hence, $\mathcal{C}_{\mathcal{X}}(f)>0$ which by Theorem \ref{lema funções em U são NUH} concludes the proof of Theorem \ref{teo. A  para nao homotetias}.

\bibliographystyle{ieeetr}
\bibliography{bibNUHend}

\end{document}